\documentclass[10pt]{amsart}

\usepackage{amssymb}
\usepackage{amsthm}
\usepackage{amsmath}

\usepackage[usenames]{color}
\definecolor{ANDREW}{RGB}{255,127,0}

\usepackage{tikz}
\usetikzlibrary{arrows,calc}

\usepackage[foot]{amsaddr}

\usepackage{hyperref}

\theoremstyle{plain}
\newtheorem{proposition}{Proposition}[section]
\newtheorem{theorem}[proposition]{Theorem}
\newtheorem{lemma}[proposition]{Lemma}
\newtheorem{corollary}[proposition]{Corollary}
\theoremstyle{definition}

\newtheorem{definition}[proposition]{Definition}
\newtheorem{observation}[proposition]{Observation}
\theoremstyle{remark}
\newtheorem{remark}[proposition]{Remark}

\DeclareMathOperator{\Aff}{Aff}
\DeclareMathOperator{\Aut}{Aut}

\DeclareMathOperator{\Imaginary}{Im}

\DeclareMathOperator{\Hol}{Hol}

\DeclareMathOperator{\arctanh}{tanh^{-1}}

\DeclareMathOperator{\Span}{Span} 
\DeclareMathOperator{\Sym}{Sym} 
 
\DeclareMathOperator{\id}{id}

\DeclareMathOperator{\Fc}{\mathcal{F}}

\DeclareMathOperator{\Hc}{\mathcal{H}}
\DeclareMathOperator{\Ic}{\mathcal{I}}

\DeclareMathOperator{\Lc}{\mathcal{L}}

\DeclareMathOperator{\Oc}{\mathcal{O}}

\DeclareMathOperator{\Cb}{\mathbb{C}}

\DeclareMathOperator{\Nb}{\mathbb{N}}

\DeclareMathOperator{\Rb}{\mathbb{R}}

\newcommand{\abs}[1]{\left|#1\right|}

\newcommand{\norm}[1]{\left\|#1\right\|}

\newcommand{\wh}[1]{\widehat{#1}}
\newcommand{\ip}[1]{\left\langle #1\right\rangle}

\begin{document}

\title[Analytic polyhedron with non-compact automorphism group]{Generic analytic polyhedron with non-compact automorphism group}
\author{Andrew Zimmer}\address{Department of Mathematics, University of Chicago, Chicago, IL 60637.}
\email{aazimmer@uchicago.edu}
\date{\today}
\subjclass[2010]{}

\begin{abstract} In this paper we prove the following rigidity theorem: a generic analytic polyhedron with non-compact automorphism group is biholomorphic to the product of a complex manifold with compact automorphism group and a polydisk. Moreover, this complex manifold and the dimension of this polydisk can be explicitly described in terms of the limit set of the automorphism group. 
\end{abstract}

\maketitle

\section{Introduction}

Given a bounded domain $\Omega \subset \Cb^d$ let $\Aut(\Omega)$ be the \emph{automorphism group} of $\Omega$, that is the group of biholomorphisms $f:\Omega \rightarrow \Omega$. The group $\Aut(\Omega)$ has a Lie group structure compatible with the compact-open topology and $\Aut(\Omega)$ acts properly on $\Omega$. It is a long standing problem to characterize the domains $\Omega$ where $\partial \Omega$ has nice properties and $\Aut(\Omega)$ is non-compact (see the survey~\cite{IK1999}). One well known result along these lines is the Wong-Rosay ball theorem:

\begin{theorem}[The Wong-Rosay Ball Theorem \cite{R1979, W1977}]
Suppose $\Omega \subset \Cb^d$ is a bounded strongly pseudoconvex domain. Then $\Aut(\Omega)$ is non-compact if and only if $\Omega$ is biholomorphic to the unit ball. 
\end{theorem}

In this paper we will consider a particular class of domains which are in some sense as far as possible from being strongly pseudoconvex:

\begin{definition} A domain $\Omega \subset \Cb^d$ is called an \emph{analytic polyhedron} if there exists a neighborhood $U$ of $\Omega$ and holomorphic functions $f_1,\dots, f_N : U \rightarrow \Cb$ so that 
\begin{align*}
\Omega = \{ z \in U : \abs{f_i(z)} < 1 \text{ for } i =1, \dots, N\}.
\end{align*}
An analytic polyhedron is called \emph{generic} if we can choose the functions $f_1, \dots, f_N$ so that whenever $\xi \in \partial \Omega$ and $\abs{f_{i_1}(\xi)} =  \dots = \abs{f_{i_r}(\xi)} = 1$, then the vectors 
\begin{align*}
\nabla f_{i_1}(\xi), \dots, \nabla f_{i_r}(\xi)
\end{align*}
are $\Cb$-linearly independent. In this case we call $f_1, \dots, f_N$ a \emph{set of generic defining functions for $\Omega$}. 
\end{definition} 

The simplest example of an analytic polyhedron with non-compact automorphism group is the polydisk:
\begin{align*}
\Delta^r = \{ (z_1, \dots, z_r) \in \Cb^r : \abs{z_1} < 1, \dots, \abs{z_r} < 1\}.
\end{align*}
Additional examples can be constructed by taking a product of a polydisk and an analytic polyhedron. When $\Omega$ is convex, Kim showed that (up to biholomorphism) these are the only examples of generic analytic polyhedron with non-compact automorphism group:

\begin{theorem}\cite{K1992} Suppose $\Omega \subset \Cb^d$ is a bounded convex generic analytic polyhedron and $\Aut(\Omega)$ is non-compact. Then there exists some $r > 0$ and a convex domain $W \subset \Cb^{d-r}$ so that $\Omega$ is biholomorphic to $\Delta^r \times W$. 
\end{theorem}

In dimension two Kim, Krantz, and Spiro removed the convexity hypothesis and gave an explicit description of $r$ and $W$:

\begin{theorem}\cite{KKS2005}\label{thm:previous_2}
Suppose $\Omega \subset \Cb^2$ is a bounded generic analytic polyhedron and $\Aut(\Omega)$ is non-compact. Then:
\begin{enumerate}
\item If an automorphism orbit accumulates at a singular boundary point, then $\Omega$ is biholomorphic to $\Delta^2$, 
\item If an automorphism orbit accumulates at a smooth boundary point, then $\Omega$ is biholomorphic to the product of $\Delta$ and the maximal analytic variety in $\partial \Omega$ passing through the orbit accumulation point. 
\end{enumerate}
\end{theorem}

\begin{remark}  Kim, Krantz, and Spiro's theorem generalized an earlier result of Kim and Pagino~\cite{KP2001} and is related to a result of Fu and Wong~\cite{FW2000}: any simply-connected domain in $\Cb^2$ with generic piecewise smooth Levi-flat boundary and non-compact automorphism group is biholomorphic to a bidisc.\end{remark}

In this paper we complete the characterization of bounded generic analytic polyhedron with non-compact automorphism group:

\begin{theorem}\label{thm:main1}
Suppose $\Omega$ is a bounded generic analytic polyhedron and $\Aut(\Omega)$ is non-compact. Then there exists some $r > 0$ and a complex manifold $W$ with $\Aut(W)$ compact so that $\Omega$ is biholomorphic to $\Delta^r \times W$. 
\end{theorem}

As in the Kim, Krantz, and Spiro result we can explicitly describe $r$ and $W$ in terms of the limit set of $\Aut(\Omega)$.

\begin{definition}
Suppose $\Omega$ is a bounded domain in $\Cb^d$. The \emph{limit set of $\Aut(\Omega)$} is the set $\Lc(\Omega) \subset \partial \Omega$ of points $\xi \in \partial \Omega$ where there exists a sequence $\varphi_n \in \Aut(\Omega)$ and a point $x \in \Omega$ with $\varphi_n(x) \rightarrow \xi$.
\end{definition}

\begin{remark} When $\Omega$ is a bounded domain $\Aut(\Omega)$ acts properly on $\Omega$ and so the set $\Lc(\Omega)$ is non-empty if and only if $\Aut(\Omega)$ is non-compact. \end{remark}

Suppose $\Omega$ is a generic analytic polyhedron and $f_1, \dots, f_N$ is a set of generic defining functions for $\Omega$. Then for a point $\xi \in \partial \Omega$ define 
\begin{align*}
 \Ic(\xi) = \{ i : \abs{f_i(\xi)} =1 \}
\end{align*}
and
\begin{align*}
r(\xi) := \#\Ic(\xi).
\end{align*}
Also let $\Fc(\xi) \subset \partial \Omega$ be the connected component of
\begin{align*}
 \left\{ \eta \in \partial \Omega : f_i(\eta) = f_i(\xi) \text{ if } i \in \Ic(\xi) \text{ and } \abs{f_i(\eta)} < 1 \text{ if } i \notin \Ic(\xi) \right\}
\end{align*}
which contains $\xi$. Then $\Fc(\xi)$ is the maximal analytic variety in $\partial \Omega$ passing through $\xi$. With this notation we will prove the following:

\begin{theorem}\label{thm:main2}
Suppose $\Omega \subset \Cb^d$ is a bounded generic analytic polyhedron. If $\xi \in \Lc(\Omega)$, then $\Omega$ is biholomorphic to $\Delta^{r(\xi)} \times \Fc(\xi)$. 
\end{theorem}

We can also explicitly describe the complex manifold $W$ in the statement of Theorem~\ref{thm:main1}.

\begin{proposition}\label{prop:main3}
Suppose $\Omega \subset \Cb^d$ is a bounded generic analytic polyhedron. If $\xi \in \Lc(\Omega)$ and 
\begin{align*}
r(\xi)= \max \{ r(\eta) : \eta \in \Lc(\Omega) \},
\end{align*}
then $\Aut(\Fc(\xi))$ is compact. 
\end{proposition}

Notice that Theorem~\ref{thm:main1} is an immediate consequence of Theorem~\ref{thm:main2} and Proposition~\ref{prop:main3}. 

\subsection{Outline of the proof of Theorem~\ref{thm:main2} and Proposition~\ref{prop:main3}:}

Like many of the results characterizing domains with large automorphism groups, a key step in our argument is rescaling. Historically rescaling has been successfully implemented only when the domain in question is 
\begin{enumerate}
\item convex (see for instance~\cite{F1989, K1992, BP1994, G1997}),
\item in $\Cb^2$ (see for instance~\cite{BP1988, BP1998, KKS2005, V2009}), or
\item strongly pseudoconvex (see for instance~\cite{P1991}). 
\end{enumerate}
In particular, none of the classical rescaling methods apply directly to analytic polyhedron. In the proof below, we make a rescaling argument work by constructing a holomorphic embedding $F:\Omega \hookrightarrow \Delta^M$, proving it has good properties, and then applying the rescaling method from the convex case to $F(\Omega) \subset \Delta^M$. 

This embedding is constructed in an obvious way: if necessary we introduce additional ``dummy''  holomorphic functions $f_{N+1}, \dots, f_M : U \rightarrow \Cb$ so that:
\begin{enumerate}
\item for any $N+1 \leq i \leq M$, $f_i(\overline{\Omega}) \subset \Delta$,
\item for any $z, w \in \overline{\Omega}$ distinct there exists some $1 \leq i \leq M$ so that $f_i(z)  \neq f_i(w)$,
\item for any point $z \in \overline{\Omega}$ we have
\begin{align*}
\Span_{\Cb} \left\{ \nabla f_1(z), \dots, \nabla f_M(z)\right\} = \Cb^d.
\end{align*}
\end{enumerate}
Then the map $F=(f_1, \dots, f_M) : U \rightarrow \Cb^M$ induces a holomorphic embedding $F:\Omega \hookrightarrow \Delta^M$. 

Proving that this embedding has good properties is more involved. This is accomplished in Sections~\ref{sec:kob_metric} and~\ref{sec:kob_dist} where the main aim is to prove estimates for the Kobayashi metric and distance showing that $F$ behaves like a quasi-isometric embedding. 

Then given an unbounded sequence $\varphi_n \in \Aut(\Omega)$ we consider the points $\overline{w}_n = F(\varphi_n( w_0)) \in \Delta^M$ and apply the  rescaling method to the sequence $(\Delta^M, \overline{w}_n)$. This produces affine maps $\overline{A}_n \in \Aff(\Cb^M)$ so that (up to reordering coordinates and passing to a subsequence) the sequence $\overline{A}_n( \Delta^M)$ converges in the local Hausdorff topology to $\Hc^r \times \Delta^{M-r}$ and $\overline{A}_n(\overline{w}_n)$ converges to a point $\overline{w}_\infty \in \Hc^r \times \Delta^{M-r}$. 

Then using the estimates on the Kobayashi metric and distance, we will show that $\left(\overline{A}_n F \varphi_n \right): \Omega \rightarrow \Cb^M$ is a normal family and after passing to a subsequence converges to a holomorphic embedding $\Phi: \Omega \rightarrow \Hc^r \times \Delta^{M-r}$. The final step in the proof of Theorem~\ref{thm:main2} is to show that the image of $\Phi$ has the form $\Hc^r \times W$ where $W$ is biholomorphic to $\Fc(\xi)$. 

To prove Proposition~\ref{prop:main3} we will use the classical theory of characteristic decompositions of analytic polyhedron to establish the following refinement of Theorem~\ref{thm:main2}:

\begin{theorem}(see Theorem~\ref{thm:main2_refined} below)
Suppose $\Omega \subset \Cb^d$ is a bounded generic analytic polyhedron. If $\xi \in \Lc(\Omega)$, then there exists a biholomorphism  $\Phi:\Omega \rightarrow \Delta^{r(\xi)} \times \Fc(\xi)$ so that: if $\theta \in \Aut(\Fc(\xi))$ and $\wh{\theta} = \Phi^{-1} (\id, \theta) \Phi$ then 
\begin{align*}
f_i\left(\wh{\theta}(z)\right) = f_i(z)
\end{align*}
for all $i \in \Ic(\xi)$ and $z \in \Omega$.
\end{theorem}

Using this refinement it is straightforward to prove the contrapositive of Proposition~\ref{prop:main3}: if $\xi \in \Lc(\Omega)$ and $\Aut(\Fc(\xi))$ is non-compact, then there exists some $\eta \in \Lc(\Omega)$ with $r(\eta) > r(\xi)$. 

\subsection{Basic notation} We now fix some very basic notations.

\begin{itemize}
\item Let $\Delta: = \{ z \in \Cb : \abs{z} < 1\}$.
\item Let $\Hc = \{ z \in \Cb : \Imaginary(z) > 0\}$.
\item Let $\ip{\cdot,\cdot}$ be the standard Hermitian inner product on $\Cb^d$ and for $z \in \Cb^d$ let $\norm{z}:=\ip{z,z}^{1/2}$.
\item Given two open sets $\Omega_1 \subset \Cb^{d_1}$ and $\Omega_2 \subset \Cb^{d_2}$ let $\Hol(\Omega_1, \Omega_2)$ be the space of holomorphic maps from $\Omega_1$ to $\Omega_2$. 
\item Given two open sets $\Oc_1 \subset \Rb^{d_1}$, $\Oc_2 \subset \Rb^{d_2}$, a $C^1$ map $F: \Oc_1 \rightarrow \Oc_2$, and a point $x \in \Oc_1$ define the derivative $d(F)_x : \Rb^{d_1} \rightarrow \Rb^{d_2}$ of $F$ at $x$ by 
\begin{align*}
d(F)_x(v) := \left.\frac{d}{dt}\right|_{t=0} F(x+tv).
\end{align*}
\item Given a domain $\Omega \subset \Cb^d$ and a holomorphic function $f:\Omega \rightarrow \Cb$ define $\nabla f(z) \in \Cb^d$ by 
\begin{align*}
d(f)_z(v) = \ip{v, \nabla f(z)} \text{ for all } v \in \Cb^d.
\end{align*}
\end{itemize}

\section{The Kobayashi metric and rescaling}

In this section we recall some basic properties of the Kobayashi metric and distance.  A more thorough discussion can be found in~\cite{K2005} or~\cite{A1989}. We will then discuss the rescaling method in the particular case of polydisks. 

\subsection{The Kobayashi metric and distance} 
Given a domain $\Omega \subset \Cb^d$ the \emph{(infinitesimal) Kobayashi metric} is the pseudo-Finsler metric
\begin{align*}
k_{\Omega}(x;v) = \inf \left\{ \abs{\zeta} : f \in \Hol(\Delta, \Omega), \ f(0) = x, \ d(f)_0(\zeta) = v \right\}.
\end{align*}
By a result of Royden~\cite[Proposition 3]{R1971} the Kobayashi metric is an upper semicontinuous function on $\Omega \times \Cb^d$. In particular if $\sigma:[a,b] \rightarrow \Omega$ is an absolutely continuous curve (as a map $[a,b] \rightarrow \Cb^d$), then the function 
\begin{align*}
t \in [a,b] \rightarrow k_\Omega(\sigma(t); \sigma^\prime(t))
\end{align*}
is integrable and we can define the \emph{length of $\sigma$} to  be
\begin{align*}
\ell_\Omega(\sigma)= \int_a^b k_\Omega(\sigma(t); \sigma^\prime(t)) dt.
\end{align*}
One can then define the \emph{Kobayashi pseudo-distance} to be
\begin{multline*}
 K_\Omega(x,y) = \inf \left\{\ell_\Omega(\sigma) : \sigma\colon[a,b]
 \rightarrow \Omega \text{ is absolutely continuous}, \right. \\
 \left. \text{ with } \sigma(a)=x, \text{ and } \sigma(b)=y\right\}.
\end{multline*}
This definition is equivalent to the standard definition of $K_\Omega$ via analytic chains, see~\cite[Theorem 3.1]{V1989}.

One of the most important properties of these constructions is the following distance decreasing property (which is immediate from the definitions):

\begin{proposition}\label{prop:dist_decrease}
Suppose $\Omega_1 \subset \Cb^{d_1}$ and $\Omega_2 \subset \Cb^{d_2}$ are domains and $f:\Omega_1 \rightarrow \Omega_2$ is a holomorphic map. Then 
\begin{align*}
k_{\Omega_2}(f(x); d(f)_x(v) ) \leq k_{\Omega_1}(x;v)
\end{align*}
and 
\begin{align*}
K_{\Omega_2}(f(x), f(y)) \leq K_{\Omega_1}(x,y)
\end{align*}
for all $x,y \in \Omega_1$ and $v \in \Cb^{d_1}$. In particular, when $\Omega \subset \Cb^d$ is a domain the group $\Aut(\Omega)$ acts by isometries on the pseudo-metric space $(\Omega, K_\Omega)$. 
\end{proposition}

\subsection{The unit disk}
Using the Schwarz lemma it is straightforward to see that the Kobayashi metric coincides with the Poincar{\'e} metric on the unit disk (at least up to a constant). In particular, 
\begin{align*}
k_{\Delta}(x; v) = \frac{\abs{v}}{1-\abs{x}^2}\end{align*}
and 
\begin{align*}
K_{\Delta}(x,y) = \arctanh \abs{\frac{x-y}{1-x\overline{y}}}
\end{align*}
for all $x,y \in \Delta$ and $v \in \Cb$. 

This explicit description implies the following localization result: 

\begin{observation}\label{obs:disk_local} Suppose $\xi \in \partial \Delta$ and $U$ is a neighborhood of $\xi$. Then for any $\delta > 0$ there exists a neighborhood $V \Subset U$ so that 
\begin{align*}
k_{U \cap \Delta}(x;v) \leq e^{\delta} k_{\Delta}(x;v)
\end{align*}
and 
\begin{align*}
K_{U \cap \Delta}(x,y) \leq e^{\delta} K_\Delta(x, y)
\end{align*}
for all $x,y \in V \cap \Delta$ and $v \in \Cb$. 
\end{observation}

\subsection{The Kobayashi metric on products}

If 
\begin{align*}
 \Omega = \Omega_1 \times \dots \times \Omega_k \subset \Cb^{d_1} \times \dots \times \Cb^{d_k}
\end{align*}
is a product of domains, it is straightforward to verify that 
\begin{align*}
k_\Omega(x;v) \leq \max_{i=1, \dots, k} k_{\Omega_i}(x_i, v_i)
\end{align*}
for all $x=(x_1, \dots, x_k) \in \Omega$ and $v=(v_1, \dots, v_k) \in \Cb^{d_1} \times \dots \times \Cb^{d_k}$. Using Proposition~\ref{prop:dist_decrease} and the natural projections, one also has 
\begin{align*}
k_\Omega(x;v) \geq \max_{i=1, \dots, k} k_{\Omega_i}(x_i, v_i)
\end{align*}
for all $x=(x_1, \dots, x_k) \in \Omega$ and $v=(v_1, \dots, v_k) \in \Cb^{d_1} \times \dots \times \Cb^{d_k}$. Thus we have the following:

\begin{observation}
With the notation above, 
\begin{align*}
k_\Omega(x;v) = \max_{i=1, \dots, k} k_{\Omega_i}(x_i, v_i)
\end{align*}
for all $x=(x_1, \dots, x_k) \in \Omega$ and $v=(v_1, \dots, v_k) \in \Cb^{d_1+\dots+d_k}$.
\end{observation}

\subsection{The Kobayashi metric on analytic polyhedron}

For analytic polyhedron the Kobayashi pseudo-distance is Cauchy complete:

\begin{proposition}\label{prop:cauchy_complete}
Suppose that $\Omega$ is a bounded analytic polyhedron. Then $(\Omega, K_\Omega)$ is a Cauchy complete metric space. 
\end{proposition}

\begin{proof}
Since $\Omega$ is bounded, $(\Omega, K_\Omega)$ is a metric space, see for instance Corollaries 2.3.2 and 2.3.6 in~\cite{A1989}. Now if $f:\Omega \rightarrow \Delta$ is holomorphic Proposition~\ref{prop:dist_decrease}  implies that
\begin{align*}
K_\Delta(f(z), f(w)) \leq K_\Omega(z,w)
\end{align*}
for all $z, w \in \Omega$. So, since $\Omega$ is an analytic polyhedron, the metric space $(\Omega, K_\Omega)$ is proper (that is, bounded sets in $(\Omega, K_\Omega)$ are relatively compact). But by definition $(\Omega, K_\Omega)$ is a length space and so by the Hopf-Rinow theorem for length spaces, see for instance Corollary 3.8 in Chapter I of~\cite{BH1999}, $(\Omega, K_\Omega)$ is Cauchy complete. 
\end{proof}
 
\subsection{Rescaling polydisks}\label{subsec:rescaling_disks} 

As mentioned in the introduction, we will use the rescaling method from the convex case in the proof of Theorem~\ref{thm:main2}. In the case of polydisks, this procedure is very explicit and in this subsection we will describe only the observations we need. For a general discussion of rescaling methods see~\cite{F1989, P1991, F1991, K2004, KK2008}.

Suppose that $w_n=\left(w_n^{(1)}, \dots, w_n^{(r)}\right)$ is a sequence in $\Delta^r$ and 
\begin{align*}
\lim_{n \rightarrow \infty} w_n =\xi \in ( \partial \Delta)^r.
\end{align*}
Next consider affine map $A_n \in \Aff(\Cb^r)$ given by 
\begin{align*}
A_n(z_1, \dots, z_r) =\left(  \lambda_n^{(1)}\left( z_1 - \frac{w_n^{(1)}}{\abs{w_n^{(1)}}} \right), \dots, \ \lambda_n^{(r)}\left( z_r - \frac{w_n^{(r)}}{\abs{w_n^{(r)}}} \right) \right)
\end{align*}
where 
\begin{align*}
\lambda_n^{(i)} = \frac{i\abs{w_n^{(i)}}}{w_n^{(i)}\left(1-\abs{w_n^{(i)}}\right)}.
\end{align*}

Then $A_n(w_n) = (i,\dots, i)$. Since
\begin{align*}
\lim_{n \rightarrow \infty} \abs{ w_n^{(i)}} = 1 \text{ for } 1 \leq i \leq r,
\end{align*}
it is easy to see that $A_n(\Delta^r)$ converges in the local Hausdorff topology to $\Hc^r$. A straight-forward computation shows that
\begin{align*}
k_{\Hc^r}(x;v) = \lim_{n \rightarrow \infty} k_{A_n \Delta^r}(x;v)
\end{align*}
 uniformly on compact sets of $\Hc^r  \times \Cb^r$ and 
 \begin{align*}
K_{\Hc^r}(x,y) = \lim_{n \rightarrow \infty} K_{A_n \Delta^r}(x,y)
\end{align*}
 uniformly on compact sets of $\Hc^r  \times \Hc^r$.

As a consequence of this discussion we have the following:

\begin{observation}\label{obs:blow_up} With the notation above, for any $u \in \Hc^r$ there exists $u_n \in \Delta^r$ so that 
\begin{enumerate}
\item $\lim_{n \rightarrow \infty} A_n(u_n) = u$,
\item $\lim_{n \rightarrow \infty} K_{\Delta^r}(u_n, w_n) = K_{\Hc^r}(u, (i,\dots, i))$.
\end{enumerate}
\end{observation}

In the context of the above observation the fact that 
\begin{align*}
\lim_{n \rightarrow \infty} K_{\Delta^r}(u_n, w_n) < \infty
\end{align*}
implies that 
\begin{align*}
\lim_{n \rightarrow \infty} \norm{u_n - w_n} = 0.
\end{align*}

\section{Local coordinates near a boundary point}

In this section we construct useful coordinates around a given boundary point of a generic analytic polyhedron:

\begin{proposition}\label{prop:chart}
Suppose that $\Omega \subset \Cb^d$ is a bounded generic analytic polyhedron with generic defining functions $f_1, \dots, f_N:U \rightarrow \Cb$. If $\xi \in \partial \Omega$ and 
\begin{align*}
\{ i_1, \dots, i_r\} =  \{ i : \abs{f_i(\xi)}=1\},
\end{align*}
then there exists a neighborhood $\Oc$ of $\xi$ and holomorphic maps $\Psi:\Oc \rightarrow \Cb^d$ and $\psi: \Oc \rightarrow \Cb^{d-r}$ so that 
\begin{enumerate}
\item $\Psi$ is a biholomorphism onto its image,
\item $\Psi(z) = (f_{i_1}(z), \dots, f_{i_r}(z), \psi(z))$ for all $z \in \Oc$, 
\item $\Psi(\Oc) = \Oc_1 \times \dots \times \Oc_d$ for some open sets $\Oc_i \subset \Cb$,
\item $\Psi(\Oc \cap \Omega) = \Psi(\Oc) \cap \Delta^d = (\Oc_1 \cap \Delta) \times \dots \times (\Oc_r \cap \Delta) \times \Oc_{r+1} \times \dots \times \Oc_d$.
\end{enumerate}
\end{proposition}

\begin{proof} For a neighborhood $\Oc$ of $\xi$ define the map $F:\Oc \rightarrow \Cb^r$ by
 \begin{align*}
F(z) = (f_{i_1}(z), \dots, f_{i_r}(z)).
\end{align*}
By shrinking $\Oc$ we may assume that 
\begin{align*}
 \Omega \cap \Oc = F^{-1}(\Delta^r \cap F(\Oc)).
\end{align*}
Since the vectors 
\begin{align*}
\nabla f_{i_1}(\xi), \dots, \nabla f_{i_r}(\xi)
\end{align*}
are $\Cb$-linearly independent we see that $d(F)_\xi$ has full rank. Hence, after possibly shrinking $\Oc$, we can find domains $V \subset \Cb^r$, $W \subset \Cb^{d-r}$ and a biholomorphism $\Psi: \Oc \rightarrow V \times W$ so that 
\begin{align*}
F = \pi_1 \circ \Psi
\end{align*}
 where $\pi_1: V \times W \rightarrow V$ is the natural projection. Thus there exists a holomorphic map  $\psi: \Oc \rightarrow \Cb^{d-r}$ so that
\begin{align*}
 \Psi(z) = (f_{i_1}(z), \dots, f_{i_r}(z), \psi(z))
\end{align*}
for all $z \in \Oc$. Moreover 
\begin{align*}
 \Psi^{-1}\left(\left(\Delta^r \cap V\right) \times W\right) = F^{-1}(\Delta^r \cap F(\Oc)) = \Omega \cap \Oc.
\end{align*}
Finally by shrinking $\Oc$ again we may assume that $\Psi(\Oc) = \Oc_1 \times \dots \times \Oc_d$ for some open sets $\Oc_i \subset \Cb$.
\end{proof}

\section{Estimates for the Kobayashi metric}\label{sec:kob_metric}

For the rest of the section suppose that $\Omega \subset \Cb^d$ is a bounded generic analytic polyhedron with generic defining functions $f_1, \dots, f_N:U \rightarrow \Cb$. 

Using Proposition~\ref{prop:dist_decrease} we immediately obtain the following lower bound for the Kobayashi metric: 

\begin{proposition}
With the notation above, 
\begin{align*}
\max_{i=1,\dots, N} k_{\Delta}(f_i(z); d(f_i)_z(v)) \leq k_\Omega(z;v).
\end{align*}
for all $z \in \Omega$ and $v \in \Cb^d$. 
\end{proposition}

In this section we will prove the following upper bound:

\begin{theorem}\label{thm:kob_metric}
With the notation above, there exists a constant $C \geq 1$ so that 
\begin{align*}
k_\Omega(z;v) \leq C \left( \norm{v}+\max_{i=1,\dots, N} k_{\Delta}(f_i(z); d(f_i)_z(v)) \right)
\end{align*}
for all $z \in \Omega$ and $v \in \Cb^d$. 
\end{theorem}

\begin{remark} With no additional assumptions on the defining functions it is necessary to have the $\norm{v}$ term in the estimate above. In particular, it is possible for there to exist some $z_0 \in \Omega$ where 
\begin{align*}
\Span_{\Cb} \left\{ \nabla f_1(z_0), \dots, \nabla f_N(z_0) \right\} \neq \Cb^d.
\end{align*}
Then there would exist some non-zero $v \in \Cb^d$ so that 
\begin{align*}
\max_{i=1,\dots, N} k_{\Delta}(f_i(z_0); d(f_i)_{z_0}(v)) = 0.
\end{align*}
\end{remark}

We begin the proof of Theorem~\ref{thm:kob_metric} by proving a local version:
\begin{lemma}
With the notation above, for any $\xi \in \partial \Omega$ there exists a neighborhood $V$ of $\xi$ and a constant $C \geq 1$ so that 
\begin{align*}
k_\Omega(z;v) \leq C \left(\norm{v} +  \max_{i \in \Ic(\xi)} k_{\Delta}(f_i(z); d(f_i)_z(v)) \right)
\end{align*}
for all $z \in \Omega \cap V$ and $v \in \Cb^d$. 
\end{lemma}

\begin{proof}
After relabeling the $f_i$, we may assume that $\abs{f_1(\xi)}=\dots = \abs{f_r(\xi)} =1$ and $\abs{f_i(\xi)} < 1$ when $i > r$. Then 
\begin{align*}
 \Ic(\xi) = \{1, \dots, r\}.
\end{align*}
By Proposition~\ref{prop:chart}, there exists a neighborhood $\Oc$ of $\xi$ and holomorphic maps $\Psi:\Oc \rightarrow \Cb^d$ and $\psi:\Oc \rightarrow \Cb^{d-r}$ so that:
\begin{enumerate}
\item $\Psi$ is a biholomorphism onto its image,
\item $\Psi(z) = (f_{1}(z), \dots, f_{r}(z), \psi(z))$ for all $z \in \Oc$, 
\item $\Psi(\Oc) = \Oc_1 \times \dots \times \Oc_d$ for some open sets $\Oc_i \subset \Cb$,
\item $\Psi(\Oc \cap \Omega) = \Psi(\Oc) \cap \Delta^d = (\Oc_1 \cap \Delta) \times \dots \times (\Oc_r \cap \Delta) \times \Oc_{r+1} \times \dots \times \Oc_d$.
\end{enumerate}
Using Observation~\ref{obs:disk_local} we can find a neighborhood $V \Subset \Oc$ of $\xi$ and some $M>0$ so that 
\begin{enumerate}
\item $\Psi(V) = V_1 \times \dots \times V_d$ for some open sets $V_i \subset \Cb$,
\item for $1 \leq i \leq r$ we have 
\begin{align*}
k_{\Oc_i \cap \Delta}(z;v) \leq e^{M} k_{\Delta}(z;v)
\end{align*}
for $z \in V_i \cap \Delta$ and $v \in \Cb$, and 
\item for $r < i$ we have 
\begin{align*}
k_{\Oc_i}(z;v) \leq e^{M}\norm{v}
\end{align*}
for $z \in V_i$ and $v \in \Cb$. 
\end{enumerate}

Now let $\Psi_i$ be the $i^{th}$ coordinate function of $\Psi$. Then for $z \in \Oc \cap \Omega$ and $v \in \Cb^d$
\begin{align*}
k_\Omega(z;v) 
& \leq k_{\Oc \cap \Omega}(z;v) \\
&= \max\left\{  \max_{i=1, \dots, r} k_{\Oc_i \cap \Delta}(\Psi_i(z); d(\Psi_i)_z(v)),  \max_{i=r+1, \dots, N} k_{\Oc_i}(\Psi_i(z); d(\Psi_i)_z(v))\right\}.
\end{align*}
For $1 \leq i \leq r$ and $z \in V \cap \Omega$ we have
 \begin{align*}
 k_{\Oc_i \cap \Delta}(\Psi_i(z); d(\Psi_i)_z(v)) = k_{\Oc_i \cap \Delta}(f_i(z); d(f_i)_z(v))  \leq e^{M} k_\Delta(f_i(z); d(f_i)_z(v)).
 \end{align*}
Since $\overline{V} \subset \Oc$ there exists $C_0\geq1$ so that 
\begin{align*}
\abs{  d(\Psi_i)_z(v)} \leq C_0\norm{v}
\end{align*}
for all $z \in V$ and $v \in \Cb^d$. So for $i > r$, $z \in V$, and $v \in \Cb^d$ we have: 
 \begin{align*}
 k_{\Oc_i }(\Psi_i(z); d(\Psi_i)_z(v)) \leq e^M \norm{d(\Psi_i)_z(v)} \leq C_0 e^M \norm{v}. 
  \end{align*}

So 
\begin{align*}
k_\Omega(z;v) \leq C_0e^M \left( \norm{v}+\max_{i=1,\dots, r} k_{\Delta}(f_i(z); d(f_i)_z(v)) \right)
\end{align*}
for all $z \in V \cap \Omega$ and $v \in \Cb^d$. 
\end{proof}

\begin{proof}[Proof of Theorem~\ref{thm:kob_metric}]
Now for every $\xi \in \partial \Omega$ there exists a neighborhood $V_\xi$ of $\xi$ and a constant $C_\xi \geq 1$ so that 
\begin{align*}
k_\Omega(z;v) \leq C_\xi\left(  \norm{v}+\max_{i=1,\dots, N} k_{\Delta}(f_i(z); d(f_i)_z(v)) \right)
\end{align*}
for all $z \in \Omega \cap V_\xi$ and $v \in \Cb^d$. Then since $\partial \Omega$ is compact, there exists $\xi_1, \dots, \xi_k \in \partial \Omega$ so that 
\begin{align*}
\partial \Omega \subset \cup_{i=1}^k V_{\xi_i}.
\end{align*}
Since $K := \Omega \setminus  \cup_{i=1}^k V_{\xi_i}$ is compact there exists $C_0 \geq 1$ so that 
\begin{align*}
k_\Omega(z;v) \leq C_0 \norm{v}
\end{align*}
for all $z \in K$ and $v \in \Cb^d$.

Finally let $C = \max\{C_{\xi_1}, \dots, C_{\xi_k}, C_0\}$. Then 
\begin{align*}
k_\Omega(z;v) \leq C \left( \norm{v}+\max_{i=1,\dots, N} k_{\Delta}(f_i(z); d(f_i)_z(v)) \right)
\end{align*}
for all $z \in \Omega$ and $v \in \Cb^d$. 

\end{proof}

\section{Estimates for the Kobayashi distance}\label{sec:kob_dist}

For the rest of the section suppose that $\Omega \subset \Cb^d$ is a bounded generic analytic polyhedron with generic defining functions $f_1, \dots, f_N:U \rightarrow \Cb$. 

Using Proposition~\ref{prop:dist_decrease} we immediately obtain the following lower bound for the Kobayashi distance: 

\begin{proposition}\label{prop:kob_dist_lower_bd}
With the notation above, 
\begin{align*}
\max_{i=1, \dots, N} K_\Delta(f_i(z), f_i(w)) \leq K_\Omega(z,w)
\end{align*}
for all $z,w \in\Omega$.
\end{proposition}

In this section we will establish an upper bound on the Kobayashi distance. 

\begin{theorem}\label{thm:kob_dist_up1}
With the notation above, for any $\xi \in \partial \Omega$ there exists a neighborhood $V$ of $\xi$ and a constant $C\geq1$ so that 
\begin{align*}
K_\Omega(z,w) \leq C\left( \norm{z-w} + \max_{i \in \Ic(\xi)} K_\Delta(f_i(z), f_i(w))\right)
\end{align*}
for all $z,w \in V \cap \Omega$.
\end{theorem}

\begin{remark} In general there will not exist a $C \geq 1$ so that the upper bound 
\begin{align*}
K_\Omega(z,w) \leq C\left( \norm{z-w} + \max_{i=1, \dots, N} K_\Delta(f_i(z), f_i(w))\right)
\end{align*}
holds for \textbf{all} $z,w \in \Omega$. The issue is that for some subset $\{i_1, \dots, i_r\} \subset \{1, \dots, N\}$ and $\zeta_1, \dots, \zeta_r \in \partial \Delta$ the set 
\begin{align*}
\left\{ \xi \in \partial \Omega : f_{i_j}(\xi)=\zeta_j \text{ for } 1 \leq j \leq r \text{ and } \abs{f_i(\xi)} < 1 \text{ if } i \notin \{i_1, \dots, i_r\} \right\}
\end{align*}
could have multiple components in $\partial \Omega$. But in this case, using Proposition~\ref{prop:faces} below, one could find $z_n, w_n \in \Omega$ so that 
\begin{align*}
\limsup_{n \rightarrow \infty} K_\Omega(z_n, w_n) = \infty.
\end{align*}
but
\begin{align*}
\limsup_{n \rightarrow \infty} \left( \max_{i=1, \dots, N} K_\Delta(f_i(z_n), f_i(w_n)) \right)< \infty.
\end{align*}
\end{remark}

\begin{proof}
After relabeling the $f_i$, we may assume that $\abs{f_1(\xi)}=\dots = \abs{f_r(\xi)} =1$ and $\abs{f_i(\xi)} < 1$ when $i > r$.  Then
\begin{align*}
 \Ic(\xi) = \{1, \dots, r\}.
\end{align*}
By Proposition~\ref{prop:chart}, there exists a neighborhood $\Oc$ of $\xi$ and holomorphic maps $\Psi:\Oc \rightarrow \Cb^d$ and $\psi:\Oc \rightarrow \Cb^{d-r}$ so that:
\begin{enumerate}
\item $\Psi$ is a biholomorphism onto its image,
\item $\Psi(z) = (f_{1}(z), \dots, f_{r}(z), \psi(z))$ for all $z \in \Oc$, 
\item $\Psi(\Oc) = \Oc_1 \times \dots \times \Oc_d$ for some open sets $\Oc_i \subset \Cb$,
\item $\Psi(\Oc \cap \Omega) = \Psi(\Oc) \cap \Delta^d = (\Oc_1 \cap \Delta) \times \dots \times (\Oc_r \cap \Delta) \times \Oc_{r+1} \times \dots \times \Oc_d$.
\end{enumerate}
Using Observation~\ref{obs:disk_local} we can find a neighborhood $V \Subset \Oc$ of $\xi$ and some $M >0$ so that 
\begin{enumerate}
\item $\Psi(V) = V_1 \times \dots \times V_d$ for some open sets $V_i \subset \Cb$,
\item for $1 \leq i \leq r$ we have 
\begin{align*}
K_{\Oc_i \cap \Delta}(z,w) \leq  e^{M}K_{\Delta}(z,w)
\end{align*}
for $z,w \in V_i \cap \Delta$, and 
\item for $r < i$ we have 
\begin{align*}
K_{\Oc_i}(z,w) \leq e^{M}\norm{z-w}
\end{align*}
for $z,w \in V_i$. 
\end{enumerate}

Now let $\Psi_i$ be the $i^{th}$ coordinate function of $\Psi$. Then for $z,w \in \Oc \cap \Omega$
\begin{align*}
K_\Omega(z,w)&  \leq K_{\Oc \cap \Omega}(z,w) \\
& = \max\left\{  \max_{i=1, \dots, r} K_{\Oc_i \cap \Delta}(\Psi_i(z),  \Psi_i(w)),  \max_{i=r+1, \dots, d} K_{\Oc_i }(\Psi_i(z),  \Psi_i(w)) \right\}
\end{align*}
If $1 \leq i \leq r$, then $\Psi_i(z) = f_i(z)$ and so 
\begin{align*}
K_{\Oc_i \cap \Delta}(\Psi_i(z), \Psi_i(w)) \leq e^{M}K_\Delta(f_i(z),f_i(w))
\end{align*}
for $z,w \in V \cap \Omega$. 
Since $\overline{V} \subset \Oc$ there exists $C_0 \geq 1$ so that 
\begin{align*}
\abs{  \Psi_i(z)-\Psi_i(w) } \leq C_0\norm{z-w}
\end{align*}
for all $z,w \in V$. So for $i > r$ and $z,w \in V$ 
 \begin{align*}
 k_{\Oc_i}(\Psi_i(z), \Psi_i(w))  \leq C_0e^{M}\norm{z-w}.
  \end{align*}
  
  So 
  \begin{align*}
K_\Omega(z,w) \leq C_0e^{M} \left(\norm{z-w}+ \max_{i=1, \dots, N} K_\Delta(f_i(z), f_i(w))  \right)
\end{align*}
for all $z,w \in V \cap \Omega$.
\end{proof} 
  
\section{The asymptotic geometry of generic analytic polyhedrons}

In this section we will prove two facts about the asymptotic geometry of bounded generic analytic polyhedron. 

For the rest of the section suppose that $\Omega \subset \Cb^d$ is a bounded generic analytic polyhedron with generic defining functions $f_1, \dots, f_N:U \rightarrow \Cb$. 

\begin{proposition}\label{prop:faces}
 With the notation above, suppose $x_n$ and $y_n$ are sequences in $\Omega$, $x_n \rightarrow \xi \in \partial \Omega$, $y_n \rightarrow \eta \in \partial\Omega$, and
\begin{align*}
 \limsup_{n \rightarrow \infty} K_\Omega(x_n, y_n) < \infty,
\end{align*}
then $\eta \in \Fc(\xi)$. 
\end{proposition}

\begin{proof}
After relabeling the $f_i$, we may assume that $\abs{f_1(\xi)}=\dots = \abs{f_r(\xi)} =1$ and $\abs{f_i(\xi)} < 1$ when $i > r$. 

By Proposition~\ref{prop:cauchy_complete} $(\Omega, K_\Omega)$ is a Cauchy complete length space. Thus every two points in $\Omega$ can be joined by a geodesic. Let $\sigma_n: [0,T_n] \rightarrow \Omega$ be a geodesic so that $\sigma_n(0)=x_n$ and $\sigma_n(T_n) = y_n$. Now fix $R > 0$ so that 
\begin{align*}
\Omega \Subset B_R := \{ z \in \Cb^d : \norm{z} < R\}.
\end{align*}
Then since $K_{B_R} \leq K_\Omega$ on $\Omega$ we see that 
\begin{align*}
K_{B_R}(\sigma_n(t), \sigma_n(s)) \leq K_\Omega(\sigma_n(t), \sigma_n(s)) = \abs{t-s}.
\end{align*}
So each $\sigma_n$ is $1$-Lipschitz when viewed as a map from $[0,T_n]$ to $(B_R, K_{B_R})$. Then since $(B_R, K_{B_R})$ is a Cauchy complete metric space, see for instance Corollary 2.3.6 in~\cite{A1989}, we can pass to a subsequence so that $\sigma_n$ converges uniformly to a curve $\sigma: [0,T] \rightarrow \overline{\Omega}$ with $\sigma(0)=\xi$ and $\sigma(T) = \eta$. Since 
 \begin{align*}
 K_\Omega(\sigma_n(t), \sigma_n(0)) = t
 \end{align*}
 using Proposition~\ref{prop:kob_dist_lower_bd} we see that 
 \begin{align*}
f_i(\sigma(t))=  \lim_{n \rightarrow \infty} f_i(\sigma_n(t)) = f_i(\xi) \text{ when } 1 \leq i \leq r
 \end{align*}
 and
  \begin{align*}
\abs{f_i(\sigma(t))} = \lim_{n \rightarrow \infty} \abs{f_i(\sigma_n(t))} <1 \text{ when } r < i.
 \end{align*}
So the image of $\sigma$ is contained in 
\begin{align*}
 \left\{ z \in \partial \Omega : f_i(z) = f_i(\xi) \text{ if } 1 \leq i \leq r \text{ and } \abs{f_i(z)} < 1 \text{ if } i >r \right\}
\end{align*}
 and hence $\eta \in \Fc(\xi)$.
\end{proof}

\begin{proposition} \label{prop:images_of_limits}
With the notation above, suppose $z_0 \in \Omega$, $\varphi_n \in \Aut(\Omega)$, $\varphi_n(z_0) \rightarrow \xi \in \partial \Omega$, and $\varphi_n$ converges locally uniformly to a holomorphic map $\varphi_\infty: \Omega \rightarrow \overline{\Omega}$. Then $\varphi_\infty(\Omega) = \Fc(\xi)$. 
\end{proposition}

\begin{proof}
Since 
\begin{align*}
\limsup_{n \rightarrow \infty} K_\Omega(\varphi_n(z_0), \varphi_n(z)) = K_\Omega(z_0, z)
\end{align*}
we see from the previous Proposition that $\varphi_\infty(\Omega) \subset \Fc(\xi)$. 

After relabeling the $f_i$, we may assume that $\abs{f_1(\xi)}=\dots = \abs{f_r(\xi)} =1$ and $\abs{f_i(\xi)} < 1$ when $i > r$. 

By Proposition~\ref{prop:chart}, for each $\eta \in \Fc(\xi)$ there exists a neighborhood $\Oc_\eta$ of $\eta$ and holomorphic maps $\Psi_\eta:\Oc_\eta \rightarrow \Cb^d$ and $\psi_\eta:\Oc_\eta \rightarrow \Cb^{d-r}$ so that:
\begin{enumerate}
\item $\Psi_\eta$ is a biholomorphism onto its image,
\item $\Psi_\eta(z) = (f_{1}(z), \dots, f_{r}(z), \psi_\eta(z))$ for all $z \in \Oc_\eta$, 
\item $\Psi(\Oc_\eta) = U_\eta \times W_\eta$ for some open sets $U_\eta \subset \Cb^r$ and $W_\eta \subset \Cb^{d-r}$,
\item $\Psi(\Oc_\eta \cap \Omega) = (U_\eta \cap \Delta^r) \times W_\eta$.
\end{enumerate}
Using Theorem~\ref{thm:kob_dist_up1} we may also assume that there exists some $C_\eta \geq 1$ so that 
\begin{align*}
K_\Omega(z,w) \leq C_\eta \left( 1 + \max_{i=1, \dots, r} K_\Delta(f_i(z), f_i(w)) \right)
\end{align*}
for $z,w \in \Oc_\eta \cap \Omega$. 

Now suppose that $\eta \in \Fc(\xi)$. To show that $\eta \in \varphi_\infty(\Omega)$ we need to find a sequence of points $y_n \in \Omega$ so that $y_n \rightarrow \eta$ and 
\begin{align*}
\liminf_{n \rightarrow \infty} K_\Omega(\varphi_n^{-1}(y_n), z_0) < \infty. 
\end{align*}
To this end, let $\sigma:[0,1] \rightarrow \Fc(\xi)$ be a curve with $\sigma(0)=\xi$ and $\sigma(1) = \eta$. Now we can find $\eta_1, \dots, \eta_m \in \Fc(\xi)$ so that 
\begin{align*}
\sigma([0,1]) \subset \cup_{j=1}^m \Oc_{\eta_j}.
\end{align*}
By relabeling and decreasing the size of our cover, we may assume that $\xi \in \Oc_{\eta_1}$ and for each $1 \leq j \leq m-1$ there exists some $0 \leq t_j \leq 1$ with  
\begin{align*}
\sigma(t_j) \in \Oc_{\eta_j} \cap \Oc_{\eta_{j+1}}
\end{align*}

Next let $u := (f_1, \dots, f_r)(\xi)$ and $u_n := (f_1, \dots, f_r)(\varphi_n(z_0))$. Then for $n$ large we have $u_n \in U_{\eta_j}$ for all $1 \leq j \leq m$.

Let $w_1\in W_{\eta_1}$ be the unique point so that 
\begin{align*}
 (u_n, w_1) = \Phi_{\eta_1}(\varphi_n(z_0))
\end{align*}
and for $2 \leq j \leq m$ let $w_j \in W_{\eta_j}$ be the unique point so that 
\begin{align*}
(u, w_j)=\Phi_{\eta_j}(\sigma(t_{j-1})).
\end{align*}
For $1 \leq j \leq m-1$ let $\overline{w}_j \in W_{\eta_j}$ be the unique point so that 
\begin{align*}
(u, \overline{w}_{j})=\Phi_{\eta_{j}}(\sigma(t_j))
\end{align*}
and let $\overline{w}_m \in W_{\eta_m}$ be the unique point so that 
\begin{align*}
(u, \overline{w}_m)= \Phi_{\eta_m}(\eta).
\end{align*}

Finally let 
\begin{align*}
y_n = \Phi_{\eta_m}^{-1}( u_n, \overline{w}_m) \in \Omega.
\end{align*}
Then by construction 
\begin{align*}
\lim_{n \rightarrow \infty} y_n= \eta
\end{align*}
and 
\begin{align*}
K_\Omega( \varphi_n(z_0), y_n) 
&\leq \sum_{j=1}^m K_{\Omega}\left( \Phi_{\eta_j}^{-1}(u_n, w_j), \Phi_{\eta_j}^{-1}(u_n, \overline{w}_j)\right)\\
& \leq C_{\eta_1} + \dots + C_{\eta_m}.
\end{align*}

So
\begin{align*}
\limsup_{n \rightarrow \infty} K_\Omega(\varphi_n^{-1}(y_n), z_0) \leq C_{\eta_1} + \dots + C_{\eta_m}.
\end{align*}
So by passing to a subsequence we can assume that $\varphi_n^{-1}(y_n)$ converges to some $y \in \Omega$. Then since $\varphi_n$ converges locally uniformly to $\varphi_\infty$ we see that 
\begin{align*}
\varphi_\infty(y) =\lim_{n \rightarrow \infty} \varphi_n(\varphi_n^{-1}(y_n)) = \eta.
\end{align*}
 Since $\eta \in \Fc(\xi)$ was arbitrary we see that $\varphi(\Omega) = \Fc(\xi)$. 
\end{proof}

\section{Proof of Theorem~\ref{thm:main2}} 

For the rest of the section suppose that $\Omega \subset \Cb^d$ is a bounded generic analytic polyhedron with generic defining functions $f_1, \dots, f_N:U \rightarrow \Cb$ and $\Aut(\Omega)$ is non-compact. 

\subsection{An embedding} 

If necessary, we can define holomorphic functions $f_{N+1}, \dots, f_M : U \rightarrow \Cb^d$ so that
\begin{enumerate}
\item for any $N+1 \leq i \leq M$, $f_i(\overline{\Omega}) \subset \Delta$,
\item for any $z, w \in \overline{\Omega}$ distinct there exists $1 \leq i \leq M$ so that $f_i(z)  \neq f_i(w)$,
\item for any point $z \in \overline{\Omega}$ we have
\begin{align*}
\Span_{\Cb} \left\{ \nabla f_1(z), \dots, \nabla f_M(z)\right\} = \Cb^d.
\end{align*}
\end{enumerate}

Now consider the map $F: \Omega \rightarrow \Delta^M$ given by 
\begin{align*}
F(z) = (f_1(z), \dots, f_M(z)).
\end{align*}
By construction this is a holomorphic embedding of $\Omega$ into $\Delta^M$. Moreover, since 
\begin{align*}
k_\Delta(f_i(z); d(f_i)_z(v)) = \frac{\abs{d(f_i)_z(v)}}{1-\abs{f_i(z)}^2} \geq \abs{d(f_i)_z(v)}
\end{align*}
we see that there exists $\epsilon > 0$ so that
\begin{align*}
k_{\Delta^M}(F(z); d(F)_z(v)) = \max_{i=1, \dots, M} k_\Delta(f_i(z); d(f_i)_z(v))  \geq \epsilon \norm{v}
\end{align*}
for all $z \in \overline{\Omega}$ and $v \in \Cb^d$.

Then using Theorem~\ref{thm:kob_metric}  there exists $C_0 \geq 1$ so that
\begin{align*}
k_{\Delta^M}(F(z); d(F)_z(v)) \leq k_\Omega(z;v) \leq C_0 k_{\Delta^M}(F(z); d(F)_z(v)).
\end{align*}
for all $z \in \Omega$ and $v \in \Cb^d$.

\subsection{Fixing our orbit} 

Suppose that $\xi \in \Lc(\Omega)$. Then there exists  $w_0 \in \Omega$ and a sequence $\varphi_n$ in $\Aut(\Omega)$ so that $\varphi_n(w_0) \rightarrow \xi$. Let
\begin{align*}
r = \#\left\{ i : f_i(\xi) = 1\right\}.
\end{align*}
By relabeling the functions $f_1, \dots, f_N$ we may assume that 
\begin{align*}
\{ 1, \dots, r\} = \left\{ i : f_i(\xi) = 1\right\}.
\end{align*}
Using Proposition~\ref{prop:images_of_limits} and possibly passing to a subsequence we can suppose that $\varphi_n$ converges locally uniformly to a holomorphic map $\varphi_\infty: \Omega \rightarrow \Fc(\xi)$. 

\subsection{Constructing affine maps}

Let $w_n = \varphi_n(w_0)$. Then for $1 \leq i \leq r$ let 
\begin{align*}
\lambda_i^{(n)} = \frac{ i}{f_i(w_n)}\frac{\abs{f_i(w_n)}}{\abs{f_i(w_n)}-1}.
\end{align*}
Then define affine maps $A_n \in \Aff(\Cb^r)$ by 
\begin{align*}
A_n(z_1, \dots, z_r) =  \left( \lambda_1^{(n)}\left( z_1 - \frac{f_1(w_n)}{\abs{f_1(w_n)}} \right), \dots, \lambda_r^{(n)}\left( z_r - \frac{f_r(w_n)}{\abs{f_r(w_n)}} \right) \right). 
\end{align*}
Next define affine maps $\overline{A}_n \in \Aff(\Cb^M)$ by
\begin{align*}
\overline{A}_n(z_1, \dots, z_M) = \left( A_n(z_1, \dots, z_r), z_{r+1}, \dots, z_M\right).
\end{align*}
Now 
\begin{align*}
(\overline{A}_nF\varphi_n)(w_0) = (i, \dots, i, f_{r+1}(w_n), \dots, f_M(w_n))
\end{align*}
and so $(\overline{A}_nF\varphi_n)(w_0)$ converges to a point $w_\infty \in \Hc^r \times \Delta^{M-r}$. 

\subsection{Normal families} 

By the discussion in Subsection~\ref{subsec:rescaling_disks} we see that $\overline{A}_n(\Delta^M)$ converges in the local Hausdorff topology to $\Hc^r \times \Delta^{M-r}$.  Moreover
\begin{align*}
k_{\Hc^r \times \Delta^{M-r}}(x;v) = \lim_{n \rightarrow \infty} k_{\overline{A}_n(\Delta^M)}(x;v)
\end{align*}
uniformly on compact sets of $(\Hc^r \times \Delta^{M-r}) \times \Cb^M$ and
\begin{align*}
K_{\Hc^r \times \Delta^{M-r}}(x,y) = \lim_{n \rightarrow \infty} K_{\overline{A}_n(\Delta^M)}(x,y)
\end{align*} 
uniformly on compact sets of $(\Hc^r \times \Delta^{M-r}) \times (\Hc^r \times \Delta^{M-r})$.

Now the map $\Phi_n:=(\overline{A}_n F\varphi_n) : \Omega \rightarrow \Cb^M$ satisfies
\begin{align*}
k_{\overline{A}_n (\Delta^M) }(\Phi_n(z); d(\Phi_n)_z(v)) \leq k_\Omega(z;v) \leq C_0 k_{\overline{A}_n (\Delta^M)}(\Phi_n(z); d(\Phi_n)_z(v))
\end{align*}
for all $z \in \Omega$ and $v \in \Cb^d$. The lower bound implies that $\Phi_n :\Omega \rightarrow \Cb^M$ is a normal family. So after passing to a  subsequence we can suppose that $\Phi_n$ converges locally uniformly to a holomorphic map $\Phi: \Omega \rightarrow \Hc^r \times \Delta^{M-r}$. 

Then  
\begin{align*}
k_{\Hc^r \times \Delta^{M-r}}(\Phi(z); d(\Phi)_z(v)) \leq k_\Omega(z;v) \leq C_0 k_{\Hc^r \times \Delta^{M-r}}(\Phi(z); d(\Phi)_z(v))
\end{align*}
for all $z \in \Omega$ and $v \in \Cb^d$. Notice that the upper bound implies that $\ker d(\Phi)_z = \{0\}$ for every $z \in \Omega$. 

\begin{lemma} $\Phi: \Omega \rightarrow \Hc^r \times \Delta^{M-r}$ is a biholomorphism onto its image. \end{lemma} 

\begin{proof} Since  $\ker d(\Phi)_z = \{0\}$ for every $z \in \Omega$, the map $\Phi$ is a local biholomorphism onto its image. To show that it is a global biholomorphism we need to show that $\Phi$ is one-to-one. So suppose that $\Phi(z_1) = \Phi(z_2)$. Then 
\begin{align*}
0 = \lim_{n \rightarrow \infty} \overline{A}_n \Big( F(\varphi_n(z_1)) \Big) - \overline{A}_n \Big(  F(\varphi_n(z_2)) \Big).
\end{align*}
Since $\norm{\overline{A}_n(z)-\overline{A}_n(w)} \geq \norm{z-w}$ for all $n$  this implies that 
\begin{align*}
 F(\varphi_\infty(z_1) )= \lim_{n \rightarrow \infty} F(\varphi_n(z_1)) = \lim_{n \rightarrow \infty} F(\varphi_n(z_2))  = F(\varphi_\infty(z_2)).
\end{align*}
Since $F$ is one-to-one on $\overline{\Omega}$ this implies that $\varphi_\infty(z_1) = \varphi_\infty(z_2)$. 

Now by Theorem~\ref{thm:kob_dist_up1} there exists $C_1 \geq 1$ so that for large $n$ we have
\begin{align*}
K_\Omega(\varphi_n(z_1), \varphi_n(z_2)) \leq C_1 \Big( \norm{\varphi_n(z_1) - \varphi_n(z_2)} + K_{\Delta^M}(F(\varphi_n (z_1)), F(\varphi_n (z_2))) \Big).
\end{align*}
Since $K_\Omega(z_1, z_2) = K_\Omega(\varphi_n(z_1), \varphi_n(z_2))$ for all $n \in \Nb$ we then see that
\begin{align*}
K_\Omega(z_1, z_2) 
&\leq \lim_{n \rightarrow \infty} C_1 \Big( \norm{\varphi_n(z_1) - \varphi_n(z_2)} + K_{\Delta^M}(F(\varphi_n (z_1)), F(\varphi_n (z_2))) \Big) \\
& = \lim_{n \rightarrow \infty} C_1 K_{\Delta^M}(F(\varphi_n (z_1)), F(\varphi_n (z_2))) \\
& =  \lim_{n \rightarrow \infty} C_1 K_{\overline{A}_n(\Delta^M)}(\Phi_n(z_1), \Phi_n(z_2)) = C_1 K_{\Hc^r \times \Delta^{M-r}}(\Phi(z_1), \Phi(z_2)) = 0.
\end{align*}
Thus $z_1 = z_2$ and so $\Phi$ is one-to-one. 
\end{proof}

\subsection{Analyzing $\Phi(\Omega)$}

Let $\pi_2: \Cb^r \times \Cb^{M-r} \rightarrow \Cb^{M-r}$ be the natural projection and  let $W = \pi_2(\Phi(\Omega))$. 

\begin{lemma} $\Phi(\Omega) = \Hc^r \times W$. \end{lemma}

\begin{proof}
Since $\Phi_n$ converges locally uniformly to $\Phi$, to show that some $y \in \Cb^M$ is contained in $\Phi(\Omega)$ it is enough to find a sequence $y_n \in \Omega$ so that:
\begin{enumerate}
\item $\liminf_{n \rightarrow \infty} K_\Omega(y_n, w_0) < \infty$,
\item $\lim_{n \rightarrow \infty} \Phi_n(y_n) = y$.
\end{enumerate}

Suppose that $w \in W$. Then there exists $z_0 \in \Omega$ and $u \in \Hc^r$ so that $\Phi(z_0) = (u,w)$. Consider the sequence $z_n = \varphi_n(z_0)$. Next let 
\begin{align*}
\eta:= \varphi_\infty(z_0) = \lim_{n \rightarrow \infty} \varphi_n (z_0) = \lim_{n \rightarrow \infty} z_n. 
\end{align*}
Then 
\begin{align*}
K_\Omega(w_n, z_n) = K_\Omega(\varphi_n (w_0), \varphi_n (z_0)) = K_\Omega(w_0, z_0)
\end{align*}
so $\eta \in \Fc(\xi)$ by Proposition~\ref{prop:faces}. In particular, 
\begin{align*}
\lim_{n \rightarrow \infty} f_i(w_n) =f_i(\xi) = f_i(\eta) = \lim_{n \rightarrow \infty} f_i(z_n) 
\end{align*}
for $1 \leq i \leq r$ and 
\begin{align*}
\abs{f_i(\eta)} = \lim_{n \rightarrow \infty}  \abs{f_i(z_n)} < 1 
\end{align*}
for $r < i$.

Now by Proposition~\ref{prop:chart}, there exists a neighborhood $\Oc$ of $\eta$ and holomorphic maps $\Psi:\Oc \rightarrow \Cb^d$ and $\psi:\Oc \rightarrow \Cb^{d-r}$ so that:
\begin{enumerate}
\item $\Psi$ is a biholomorphism onto its image,
\item $\Psi(z) = (f_{1}(z), \dots, f_{r}(z), \psi(z))$ for all $z \in \Oc$, 
\item $\Psi(\Oc) = V \times W$ for some open sets $V\subset \Cb^r$ and $W \subset \Cb^{d-r}$,
\item $\Psi(\Oc \cap \Omega) = \Psi(\Oc) \cap \Delta^d = (V \cap \Delta^r) \times W$.
\end{enumerate}
Using Theorem~\ref{thm:kob_dist_up1} we may also assume that there exists some $C_2 \geq 1$ so that 
\begin{align*}
K_\Omega(z,w) \leq C_2 \left( 1 + \max_{i=1, \dots, r} K_\Delta(f_i(z), f_i(w)) \right)
\end{align*}
for $z,w \in \Oc \cap \Omega$. 

Now suppose that $u^\prime \in \Hc^r$. Then by Observation~\ref{obs:blow_up} there exists $u_n^\prime \in \Delta^r$ so that 
\begin{enumerate}
\item $\lim_{n \rightarrow \infty} A_nu_n^\prime = u^\prime$,
\item $\limsup_{n \rightarrow \infty} K_{\Delta^r}\left(u_n^\prime, (f_1, \dots, f_r)(w_n) \right) < \infty$.
\end{enumerate}
Then 
\begin{align*}
\lim_{n \rightarrow \infty} u_n^\prime = \lim_{n \rightarrow \infty} (f_1, \dots, f_r)(w_n)= \lim_{n \rightarrow \infty} (f_1, \dots, f_r)(z_n)
\end{align*}
So for $n$ large, $u_n^\prime \in V$. Then let 
\begin{align*}
z_n^\prime = \Psi^{-1}( u_n^\prime, \psi(z_n)).
\end{align*}

Then 
\begin{align*}
K_\Omega(z_n, z_n^\prime)
& \leq C_2\left( 1+  K_{\Delta^r}\left( (f_1, \dots, f_r)(z_n), u_n^\prime \right) \right) \\
& \leq C_2\left(1+ K_{\Delta^r}\left( (f_1, \dots, f_r)(z_n), (f_1, \dots, f_r)(w_n) \right)+K_{\Delta^r}\left((f_1, \dots, f_r)(w_n), u_n^\prime \right)\right)
\end{align*}
and so 
\begin{align*}
\limsup_{n \rightarrow \infty} K_\Omega(z_0, \varphi_n^{-1} (z_n^\prime)) = \limsup_{n \rightarrow \infty} K_\Omega(z_n, z_n^\prime) < \infty.
\end{align*}
So after passing to a subsequence we can suppose $\varphi_n^{-1} (z_n^\prime) \rightarrow y \in \Omega$. Then since $\Phi_n$ converges locally uniformly to $\Phi$ we have
\begin{align*}
\Phi(y) &= \lim_{n \rightarrow \infty} \Phi_n(\varphi_n^{-1} (z_n^\prime))=\lim_{n \rightarrow \infty} (\overline{A}_n F)(z_n^\prime) \\
&= \lim_{n \rightarrow \infty} \left( \overline{A}_n(u_n^\prime), f_{r+1}(z_n^\prime), \dots, f_M(z_n^\prime)\right) = (u^\prime,w).
\end{align*}
Notice that $w = \lim_{n \rightarrow \infty} (f_{r+1}, \dots, f_M)(z_n^\prime)$ since $\lim_{n \rightarrow \infty} \norm{z_n - z_n^\prime} = 0$.

Since $u^\prime \in \Hc^r$ was arbitrary we see that $\Hc^r \times \{w \} \subset \Phi(\Omega)$. Then since $w \in W$ was arbitrary we see that $\Hc^r \times W = \Phi(\Omega)$. 

\end{proof}

\begin{lemma} $W$ is biholomorphic to $\Fc(\xi)$. \end{lemma}

\begin{proof} Let $G: \Fc(\xi) \rightarrow \Delta^{M-r}$ be given by $G(z) = (f_{r+1}, \dots, f_M)(z)$. Since
\begin{align*}
\Span_{\Cb} \left\{ \nabla f_1(z), \dots, \nabla f_M(z)\right\} = \Cb^d
\end{align*}
for any point $z \in \overline{\Omega}$ and the tangent space of $\Fc(\xi)$ at $z$ is the orthogonal complement of 
\begin{align*}
\Span_{\Cb} \left\{ \nabla f_1(z), \dots, \nabla f_r(z)\right\}
\end{align*}
we see that $\ker d(G)_z = \{0\}$. Thus $G$ is a local biholomorphism. But then using the fact that $F$ is one-to-one on $\overline{\Omega}$ we see that $G$ is one-to-one on $\Fc(\xi)$. Hence $G$ is a biholomorphism onto its image. 

We claim that $G(\Fc(\xi)) = W$.  Let $\pi_2 : \Cb^r \times \Cb^{M-r} \rightarrow \Cb^{M-r}$ be the projection onto the second factor. Then 
 \begin{align*}
\pi_2(\Phi(z))
& = \lim_{n \rightarrow \infty} \pi_2((\overline{A}_nF\varphi_n)(z)) = \lim_{n \rightarrow \infty} (\pi_2F)(\varphi_n(z))\\
&=  \lim_{n \rightarrow \infty} G(\varphi_n(z)) = G(\varphi_\infty(z)).
\end{align*}
So $W =   \pi_2( \Phi(\Omega)) = G(\varphi_\infty(\Omega)) = G(\Fc(\xi))$ by Proposition~\ref{prop:images_of_limits}.
\end{proof}

\section{Proof of Proposition~\ref{prop:main3}}

In this section we prove Proposition~\ref{prop:main3}. The key step is proving the following refinement of Theorem~\ref{thm:main2}:

\begin{theorem}\label{thm:main2_refined}
Suppose $\Omega \subset \Cb^d$ is a bounded generic analytic polyhedron. If $\xi \in \Lc(\Omega)$, then there exists a biholomorphism  $\Phi:\Omega \rightarrow \Delta^{r(\xi)} \times \Fc(\xi)$ so that: if $\theta \in \Aut(\Fc(\xi))$ and $\wh{\theta} = \Phi^{-1} (\id, \theta) \Phi$ then 
\begin{align*}
f_i\left(\wh{\theta}(z)\right) = f_i(z)
\end{align*}
for all $i \in \Ic(\xi)$ and $z \in \Omega$.
\end{theorem}

Delaying the proof of Theorem~\ref{thm:main2_refined} let us use it to prove Proposition~\ref{prop:main3}:

\begin{proof}[Proof of Proposition~\ref{prop:main3}]
Suppose that $\xi_0 \in \Lc(\Omega)$ and 
\begin{align*}
r(\xi_0) =  \max \{ r(\xi) : \xi \in \Lc(\Omega) \}.
\end{align*}
Let $r_0 := r(\xi_0)$. By relabeling our defining functions we can assume that 
\begin{align*}
\abs{f_1(\xi_0)} = \dots =\abs{f_{r_0}(\xi_0)} = 1
\end{align*}
and $\abs{f_i(\xi_0)} < 1$ when $i > r_0$. 

By Theorem~\ref{thm:main2_refined} there exists a biholomorphism  $\Phi:\Omega \rightarrow \Delta^{r_0} \times \Fc(\xi_0)$ so that: if $\theta \in \Aut(\Fc(\xi_0))$ and $\wh{\theta} = \Phi^{-1} (\id, \theta) \Phi$ then 
\begin{align*}
f_i\left(\wh{\theta}(z)\right) = f_i(z)
\end{align*}
for all $1 \leq i \leq r_0$ and $z \in \Omega$.

Now suppose for a contradiction that $\Aut(\Fc(\xi_0))$ is non-compact. Then there exists $\theta_n \in \Aut(\Fc(\xi_0))$ so that $\theta_n \rightarrow \infty$. Let $\wh{\theta}_n = \Phi^{-1} (\id, \theta_n) \Phi$, then $\wh{\theta}_n \rightarrow \infty$ in $\Aut(\Omega)$. 

Now fix some sequence $\varphi_n \in \Aut(\Omega)$ and some $z_0 \in \Omega$ so that $\varphi_n(z_0) \rightarrow \xi_0$. Since $\wh{\theta}_n \rightarrow \infty$ in $\Aut(\Omega)$ there exists some $1 \leq i_0 \leq N$ so that 
\begin{align*}
\limsup_{n \rightarrow \infty} \abs{f_{i_0}\left(\wh{\theta}_n (z_0)\right)} = 1.
\end{align*}
Since $f_i\left(\wh{\theta}_n (z_0)\right) = f_i(z_0)$ for $1 \leq i \leq r_0$ we see that $i_0 > r_0$. 

Now for each $k \in \Nb$ pick $n_k$ so that 
\begin{align*}
\min_{i=1, \dots, r_0} \abs{f_i(\varphi_{n_k} (z_0))} > 1-1/k 
\end{align*}
and then pick $m_k$ so that 
\begin{align*}
\abs{f_{i_0}\left(\wh{\theta}_{m_k} \varphi_{n_k} (z_0)\right)} > 1-1/k.
\end{align*}
This second choice is possible because of Proposition~\ref{prop:faces}.

Notice that 
\begin{align*}
\min_{i=1, \dots, r_0} \abs{f_i\left(\wh{\theta}_{m_k}\varphi_{n_k} (z_0)\right)} = \min_{i=1, \dots, r_0} \abs{f_i(\varphi_{n_k} (z_0))} > 1-1/k.
\end{align*}
Then by passing to a subsequence we can suppose that $\wh{\theta}_{m_k}\varphi_{n_k} (z_0) \rightarrow \eta \in \partial \Omega$. Then $\eta \in \Lc(\Omega)$ and 
\begin{align*}
\Ic(\eta) \supset \{1,\dots, r_0, i_0\}.
\end{align*}
So 
\begin{align*}
r(\eta) > r_0 = \max \{ r(\xi) : \xi \in \Lc(\Omega) \}.
\end{align*} 
and we have a contradiction. 
\end{proof}

\subsection{Characteristic decompositions} Before starting the proof of Theorem~\ref{thm:main2_refined} we will need to recall some classical facts about the characteristic decompositions of an analytic polyhedron.

Suppose that $\Omega \subset \Cb^d$ is a domain and $f:\Omega \rightarrow \Cb$ is holomorphic. Then for $z \in \Omega$ let $L(z,f)$ be the connected component of $f^{-1}(f(z))$ which contains $z$. In the case in which $\Omega$ is an analytic polyhedron with defining functions $f_1, \dots, f_N : U \rightarrow \Delta$ the decomposition of $\Omega$ into sets of the form $L(z,f_i)$ is a called a \emph{characteristic decomposition} of $\Omega$. 

The following theorem is classical:

\begin{theorem}\cite[Satz 14]{RS1960}\label{thm:char_decomp}
Suppose that $\Omega$ is an analytic polyhedron and $f_1, \dots, f_N : U \rightarrow \Omega$ is a minimal set of defining  functions (that is, any proper subset of $f_1, \dots, f_n$ is not a set of defining functions for $\Omega$). If $\varphi \in \Aut(\Omega)$, then there exists a map $\sigma: \{1, \dots, n\} \rightarrow \{1, \dots, n\}$ so that 
\begin{align*}
\varphi(L(z,f_i)) = L\left(\varphi(z), f_{\sigma(i)}\right)
\end{align*}
for all $z \in \Omega$ and $1 \leq i \leq N$. 
\end{theorem}

Because the argument is short we will provide the proof of Theorem~\ref{thm:char_decomp} in Appendix~\ref{app:proof_of_char_decomp}.

It is possible for different $f_i$ to generate the same characteristic decomposition and so in general the map $\sigma$ is not uniquely defined. This lack of uniqueness can be overcome by considering a special subset. In particular, fix a subset 
\begin{align*}
\{i_1, \dots, i_{n_0}\} \subset \{1, \dots, N\}
\end{align*}
 so that: 
 \begin{enumerate}
 \item for any $1 \leq i \leq N$ there exists a $1 \leq k\leq n_0$ so that 
\begin{align*}
L(z, f_i) = L\left(z, f_{i_k}\right)
\end{align*}
for all $z \in \Omega$,
\item for all $1 \leq k_1 < k_2 \leq n_0$ there exists some $z \in \Omega$ so that 
\begin{align*}
L\left(z, f_{i_{k_1}}\right) \neq L\left(z, f_{i_{k_2}}\right).
\end{align*}
\end{enumerate}

Then for any element $\varphi \in \Aut(\Omega)$ there exists a unique map $\sigma(\varphi) : \{1, \dots, n_0\} \rightarrow \{1, \dots, n_0\}$ so that 
\begin{align*}
\varphi\left(L(z,f_{i_k})\right) = L\left(\varphi(z), f_{\sigma(\varphi)(i_k)}\right)
\end{align*}
for any $z \in \Omega$ and $1 \leq k \leq n_0$. Using the uniqueness of $\sigma(\varphi)$, it is straight-forward to verify that $\sigma$ is a homomorphism from $\Aut(\Omega)$ to the symmetric group $\Sym(n_0)$ on $n_0$ elements. This leads to the following corollary:

\begin{corollary}\label{cor:finite_index}
Suppose that $\Omega$ is an analytic polyhedron and $f_1, \dots, f_N : U \rightarrow \Omega$ is a minimal set of defining  functions. Then there exists a finite index normal subgroup $H \leq \Aut(\Omega)$ so that 
 \begin{align*}
\varphi(L(z,f_i)) = L(\varphi(z), f_{i})
\end{align*}
for all $\varphi \in H$, $z \in \Omega$, and $1 \leq i \leq N$. 
\end{corollary}

\subsection{Proof of Theorem~\ref{thm:main2_refined}}

Suppose $\Omega$ is a generic analytic polyhedron and $f_1, \dots, f_N : U \rightarrow \Omega$ is a set of generic defining functions. We may assume that for every $1 \leq i \leq N$ the set 
\begin{align*}
\{ \xi \in \partial \Omega : \abs{f_i(\xi)} = 1 \}
\end{align*}
is non-empty. Then Proposition~\ref{prop:chart} implies that $f_1, \dots, f_N$ is a minimal set of defining functions. Then, by Corollary~\ref{cor:finite_index}, there exists a finite index normal subgroup $H \leq \Aut(\Omega)$ so that 
 \begin{align*}
\varphi(L(z,f_i)) = L(\varphi(z), f_{i})
\end{align*}
for all $\varphi \in H$, $z \in \Omega$, and $1 \leq i \leq N$.

Now fix some $\xi \in \Lc(\Omega)$. Let $r = r(\xi)$ and by  relabeling the functions $f_1, \dots, f_N$ we may assume that 
\begin{align*}
\{ 1, \dots, r\} = \left\{ i : \abs{f_i(\xi)} = 1\right\}.
\end{align*}

\begin{lemma} There exists $w_0 \in \Omega$ and a sequence $\varphi_n \in H$ so that
\begin{align*}
\varphi_n(w_0) \rightarrow \xi.
\end{align*}
\end{lemma}

\begin{proof}
Fix some $w_0^\prime \in \Omega$ and a sequence $\phi_n \in \Aut(\Omega)$ so that 
\begin{align*}
\phi_n(w_0^\prime) \rightarrow \xi.
\end{align*}
By possibly passing to a subsequence we can suppose that $\phi_n$ converges locally uniformly to a holomorphic map $\phi_\infty: \Omega \rightarrow \overline{\Omega}$. 

Consider the natural homomorphism $\rho:\Aut(\Omega) \rightarrow \Aut(\Omega)/H$ and let $n_0$ be the order of $\Aut(\Omega)/H$. By passing to a subsequence we can suppose that $\rho(\phi_n)$ is constant. Then let 
\begin{align*}
\theta = \prod_{i=1}^{n_0-1} \varphi_i 
\end{align*}
and $\varphi_n = \phi_n \theta$. Then $\varphi_n$ converges locally uniformly to $\phi_\infty\theta$ and so $\varphi_n(w_0) \rightarrow \xi$ where $w_0 = \theta^{-1}(w_0^\prime)$. Moreover, 
\begin{align*}
\rho(\varphi_n) = \rho(\phi_n) \prod_{i=1}^{n_0-1} \rho(\phi_i) =  \rho(\phi_n)^{n_0} = \id
\end{align*}
and so $\varphi_n \in H$. 
\end{proof}

Now we repeat the proof of Theorem~\ref{thm:main2} with the point $w_0 \in \Omega$ and the  sequence $\varphi_n \in H$. In particular: if necessary, we can define holomorphic functions $f_{N+1}, \dots, f_M : U \rightarrow \Cb^d$ so that
\begin{enumerate}
\item for any $N+1 \leq i \leq M$, $f_i(\overline{\Omega}) \subset \Delta$,
\item for any $z, w \in \overline{\Omega}$ distinct there exists $1 \leq i \leq M$ so that $f_i(z)  \neq f_i(w)$,
\item for any point $z \in \overline{\Omega}$ we have
\begin{align*}
\Span_{\Cb} \left\{ \nabla f_1(z), \dots, \nabla f_M(z)\right\} = \Cb^d.
\end{align*}
\end{enumerate}
Now consider the map $F: \Omega \rightarrow \Delta^M$ given by 
\begin{align*}
F(z) = (f_1(z), \dots, f_M(z)).
\end{align*}
Then let $w_n = \varphi_n(w_0)$. Then for $1 \leq i \leq r$ let 
\begin{align*}
\lambda_i^{(n)} = \frac{ i}{f_i(w_n)}\frac{\abs{f_i(w_n)}}{\abs{f_i(w_n)}-1}.
\end{align*}
Then define affine maps $A_n \in \Aff(\Cb^r)$ by 
\begin{align*}
A_n(z_1, \dots, z_r) =  \left( \lambda_1^{(n)}\left( z_1 - \frac{f_1(w_n)}{\abs{f_1(w_n)}} \right), \dots, \lambda_r^{(n)}\left( z_r - \frac{f_r(w_n)}{\abs{f_r(w_n)}} \right) \right). 
\end{align*}
Finally define affine maps $\overline{A}_n \in \Aff(\Cb^M)$ by
\begin{align*}
\overline{A}_n(z_1, \dots, z_M) = \left( A_n(z_1, \dots, z_r), z_{r+1}, \dots, z_M\right).
\end{align*}

As in the proof of Theorem~\ref{thm:main2}, we can pass to a subsequence so that the maps $\Phi_n = \overline{A}_nF \varphi_n : \Omega \rightarrow \Hc^r \times \Delta^{M-r}$ converge to a biholomorphism $\Phi : \Omega \rightarrow \Hc^r \times W$ where $W$ is biholomorphic to $\Fc(\xi)$.

Now for $1 \leq i \leq r$ let $\Phi_n^{(i)}$ be the $i^{th}$ coordinate function of $\Phi_n$ and let $\Phi^{(i)}$ be the $i^{th}$ coordinate function of $\Phi$. Consider the affine maps $\ell_n^{(i)} \in \Aff(\Cb)$ given by
\begin{align*}
\ell_n^{(i)}(z) = \lambda_i^{(n)}\left( z - \frac{f_i(w_n)}{\abs{f_i(w_n)}} \right).
\end{align*}
Then $\Phi_n^{(i)} = \ell_n^{(i)} \circ f_i \circ \varphi_n$. 

Now if $w \in L(z, f_i)$ and $1 \leq i \leq r$ then 
\begin{align*}
\varphi_n(w) \in \varphi_n( L(z, f_i)) = L(\varphi_n(z), f_i)
\end{align*}
and so $f_i(\varphi_n( w)) = f_i(\varphi_n (z))$. Then we have
\begin{align*}
\Phi^{(i)}(w) = \lim_{n \rightarrow \infty} ( \ell_n^{(i)} \circ f_i \circ \varphi_n)(w) =  \lim_{n \rightarrow \infty} ( \ell_n^{(i)} \circ f_i \circ \varphi_n)(z) = \Phi^{(i)}(z).
\end{align*}
Thus 
\begin{align*}
L(z, f_i) \subset L\left(z, \Phi^{(i)}\right)
\end{align*}
for all $z \in \Omega$ and $1 \leq i \leq r$. This implies that 
\begin{align*}
\nabla f_i(z) \wedge \nabla \Phi^{(i)}(z)=0
\end{align*}
for all $z \in \overline{\Omega}$. So by Lemma~\ref{lem:wedge} we have: 
\begin{align*}
L(z, f_i) = L\left(z, \Phi^{(i)}\right)
\end{align*}
for all $z \in \Omega$ and $1 \leq i \leq r$. 

Finally suppose that $\theta \in \Aut(\Fc(\xi))$ and $\wh{\theta} = \Phi^{-1} (\id, \theta) \Phi$. Then 
 \begin{align*}
\wh{\theta}\left( L \left(z, \Phi^{(i)}\right) \right)= L\left(z, \Phi^{(i)}\right)
 \end{align*}
 for $1 \leq i \leq r$ and $z \in \Omega$. Hence
  \begin{align*}
\wh{\theta}( L (z, f_i)) = L(z, f_i)
\end{align*}
and
  \begin{align*}
f_i\left(\wh{\theta}(z)\right) =  f_i(z)
 \end{align*}
for all $1 \leq i \leq r$ and $z \in \Omega$. 

\appendix 

\section{Characteristic decompositions}\label{app:proof_of_char_decomp}

Given a domain $\Omega \subset \Cb^d$ and a function $f:\Omega \rightarrow \Cb^d$ let $L(z,f)$ be the connected component of $f^{-1}(f(z))$ which contains $z$. In this appendix we provide a proof of the following classical result:

\begin{theorem}\cite[Satz 14]{RS1960}\label{thm:char_decomp_app}
 Suppose $\Omega$ is an analytic polyhedron and $f_1, \dots, f_N : U \rightarrow \Omega$ is a minimal defining set. Then for any $\varphi \in \Aut(\Omega)$ there exists a map $\sigma:\{1, \dots, N\} \rightarrow \{1, \dots, N\} $ so that
\begin{align*}
 \varphi(L(z,f_i)) = L(\varphi(z), f_{\sigma(i)})
\end{align*}
for all $z \in \Omega$ and $1 \leq i \leq N$. 
\end{theorem}

We begin by proving the following lemma:

\begin{lemma}\label{lem:wedge}
Suppose $\Omega \subset \Cb^d$ is a domain and $f,g: \Omega \rightarrow \Cb$ are non-constant holomorphic functions. If 
\begin{align*}
\nabla f(z) \wedge \nabla g(z) = 0
\end{align*}
for all $z \in \Omega$, then 
\begin{align*}
L(z,f) = L(z,g)
\end{align*}
for all $z \in \Omega$. 
\end{lemma}

\begin{proof}
Fix some $z_0 \in \Omega$ and let $A$ be an irreducible component of $L(z_0,f)$. Then let 
\begin{align*}
A_0 = A \cap L(z_0, f)_{reg}
\end{align*}
where $L(z_0, f)_{reg}$ is the regular points of $L(z_0, f)$. By the Theorem in~\cite[Chapter 1.5.4]{C1989}, $A_0$ is connected and open, dense in $A$. 

We claim that $A \subset L(z, g)$ for any $z \in A$. Since $A$ is an irreducible component of $L(z_0,f)$ either $d(g)_z \equiv 0$ on $A$ or $\{ z \in A : d(g)_z \neq 0\}$ is an open, dense subset of $A$. In the first case, clearly $g$ is constant on $A$ and hence $A \subset L(z, g)$ for any $z \in A$. So assume that $\{ z \in A : d(g)_z \neq 0\}$ is an open, dense subset of $A$ and fix some $w_0 \in A_0$ so that $d(g)_{w_0} \neq 0$. Then, since $\nabla f(z)$ and $\nabla g(z)$ are co-linear, there exists a neighborhood $W$ of $w_0$ and a function $h:W \rightarrow \Cb$ so that 
\begin{align*}
d(f)_w = h(w) d(g)_w \text{ for all } w \in W.
\end{align*}
This implies that $L(w, g|_W) \subset L(w, f|_W)$ for all $w \in W$. Now by assumption $w_0$ is a regular point of $L(w_0, f) = L(z_0, f)$ and so after possibly shrinking $W$ we may assume that $L(w_0, f|_W)$ is a submanifold of $W$. Since $L(w_0, g|_W) \subset L(w_0, f|_W)$ Proposition 2 in~\cite[Chapter 1.2.3]{C1989} implies that $w_0$ is a regular point of $L(w_0, g|_W)$. Hence, by possibly shrinking $W$ again, we must have that $L(w_0, g|_W) = L(w_0, f|_W)$. But this implies that $g$ is constant on $A_0 \cap W= L(w_0, f|_W)$. So, since $A$ is irreducible, we see that $g$ is constant on $A$. Thus $A \subset L(w_0, g)$. Since $A$ is connected, $A \subset L(z,g)$ for any $z \in A$. 

Then by the local Noetherian property for analytic sets, see the Theorem in~\cite[Chapter 1.5.4]{C1989}, there exists a countable set $\{z_\alpha \}$ so that 
\begin{align*}
L(z_0,f) \subset \cup L(z_\alpha, g).
\end{align*}
Then $g$ takes on at most countable many values on $L(z_0,f)$ so connectivity implies that $g$ is constant on $L(z_0,f)$ and hence $L(z_0,f) \subset L(z_0,g)$.

Since $z_0 \in \Omega$ was arbitrary we see that 
 \begin{align*}
L(z,f) \subset L(z,g)
\end{align*}
for all $z \in \Omega$. Then by reversing the role of $f$ and $g$ in the argument above we see that 
 \begin{align*}
L(z,f) = L(z,g)
\end{align*}
for all $z \in \Omega$. 
\end{proof}

The following argument is the proof of Theorem 1.3 in~\cite{Z1998} taken essentially verbatim: 

\begin{proof}[Proof of Theorem~\ref{thm:char_decomp_app}]
Fix some $1 \leq i_0 \leq N$. We will find an $1 \leq i_0^* \leq N$ so that 
\begin{align*}
\varphi\left(L(z,f_{i_0}) \right)=L\left(\varphi(z),f_{i_0^*}\right)
\end{align*}
for all $z \in \Omega$.

Since the defining set is minimal there exists some $\xi^\prime \in \partial \Omega$ so that 
\begin{align*}
\{i_0\} = \{ i : \abs{f_i(\xi^\prime)}=1\}.
\end{align*}
Then there exists a neighborhood $\Oc$ of $\xi^\prime$ so that for all $\eta \in \Oc$ we have 
\begin{align*}
\{i_0\} = \{ i : \abs{f_i(\eta)}=1\}.
\end{align*}
Since $f_{i_0}$ is holomorphic the set 
\begin{align*}
\{ u \in \Oc : \nabla f_{i_0}(u) \neq 0\} 
\end{align*}
is connected, see Proposition 3 in~\cite[Chapter 1.2.2]{C1989}. Thus there exists some $\xi \in \partial \Omega$ so that 
\begin{align*}
\{i_0\} = \{ i : \abs{f_i(\xi)}=1\}
\end{align*}
and $\nabla f_{i_0}(\xi) \neq 0$. 

Then using the proof of Proposition~\ref{prop:chart}  there exists a domain $W_1 \subset \Cb$ containing $f_{i_0}(\xi)$, a domain $W_2 \subset \Cb^{d-1}$, a neighborhood $\Oc$ of $\xi$, and a biholomorphism $\Psi: W_1 \times W_2 \rightarrow \Oc$ so that $\Psi^{-1}(\Omega) = (\Delta \cap W_1) \times W_2$ and the $1^{st}$ coordinate function of $\Psi^{-1}$ is $f_{i_0}$.

Next consider the holomorphic map $G: (\Delta \cap W_1)\times W_2 \rightarrow \otimes_{j=1}^N \Cb^{d-1}$ given by 
\begin{align*}
G(z_1, z_2) = \otimes_{j=1}^N \frac{\partial(f_j \circ \varphi \circ \Psi)}{\partial z_2}(z_1,z_2).
\end{align*}
We claim that 
\begin{align*}
\lim_{z \rightarrow \eta} G(z) = 0
\end{align*}
for any $\eta \in (\partial \Delta \cap W_1) \times W_2$. Suppose not, then there exists 
\begin{align*}
\eta = (\eta_1, \eta_2) \in (\partial \Delta \cap W_1) \times W_2
\end{align*}
and $z^{(n)}= \left(z_1^{(n)}, z_2^{(n)}\right) \rightarrow \eta$ so that 
\begin{align*}
\lim_{n \rightarrow \infty} G\left(z^{(n)}\right) \neq 0.
\end{align*}
Consider the sequence of holomorphic maps $\overline{\varphi}_n: W_2 \rightarrow \Omega$ given by
\begin{align*}
\overline{\varphi}_n(z) = \varphi\left(\Psi\left(z_1^{(n)}, z\right)\right).
\end{align*}
Since $\Omega$ is bounded we can pass to a subsequence and assume that $\overline{\varphi}_n$ converges locally uniformly to a holomorphic map $\overline{\varphi}:W_2 \rightarrow \overline{\Omega}$. Since $\varphi$ is an automorphism and hence proper, we see that $\overline{\varphi}(W_2) \subset \partial \Omega$. Then since 
\begin{align*}
\partial \Omega  \subset \cup_{j=1}^N f_j^{-1}(\partial \Delta)
\end{align*}
there exists an open set $W_2^\prime \subset W_2$ and some $1 \leq j_0 \leq N$ so that $\overline{\varphi}(W_2^\prime) \subset f_{j_0}^{-1}(\partial \Delta)$. Now $f_{j_0} \circ \overline{\varphi}: W_2 \rightarrow \Cb$ is holomorphic  and $(f_{j_0} \circ \overline{\varphi})(W_2^\prime) \subset \partial \Delta$, so $f_{j_0} \circ \overline{\varphi}: W_2 \rightarrow \Cb$ must be constant. Thus 
\begin{align*}
\frac{\partial (f_{j_0} \circ \overline{\varphi})}{\partial z}(z) = 0
\end{align*}
but then
\begin{align*}
0=\frac{\partial (f_{j_0} \circ \overline{\varphi})}{\partial z}(z_2) = \lim_{n \rightarrow \infty}  \frac{\partial(f_{j_0} \circ \varphi \circ \Psi)}{\partial z_2}\left(z_1^{(n)},z_2^{(n)}\right)
\end{align*}
which is a contradiction. Thus 
\begin{align*}
\lim_{z \rightarrow \eta} G(z) = 0
\end{align*}
for any $\eta \in (\partial \Delta \cap W_1) \times W_2$. 

Next we extend $G$ to a map on all of $W_1 \times W_2$ by defining $G(z_1, z_2) = 0$ when $z_1 \in (W_1 \setminus \Delta) \times W_2$. Then the above claim shows that $G$ is continuous and by definition $G$ is holomorphic on $\{ z : G(z) \neq 0\}$. So by Rado's Theorem, see~\cite[pg. 51]{N1971}, $G$ is holomorphic. But $G$ vanishes on an open set, namely $ (W_1 \setminus \Delta) \times W_2$, and hence is identically zero. So there exists   $1 \leq i_0^* \leq N$ so that 
\begin{align*}
\frac{\partial(f_{i_0^*} \circ \varphi \circ \Psi)}{\partial z_2}(z_1,z_2) =0
\end{align*}
on $(\Delta \cap W_1) \times W_2$. 

Then $\nabla f_{i_0} \wedge \nabla (f_{i_0^*} \circ \varphi)= 0$ on $U$ and so by analyticity $\nabla f_{i_0} \wedge \nabla (f_{i_0^*} \circ \varphi)= 0$ on $\Omega$. So by the Lemma 
\begin{align*}
L(z,f_{i_0})=L(z,f_{i_0^*} \circ \varphi)
\end{align*}
for all $z \in \Omega$. Thus 
\begin{align*}
\varphi(L(z,f_{i_0}))=L(\varphi(z),f_{i_0} \circ \varphi^{-1})=L(\varphi(z),f_{i_0^*})
\end{align*}
for all $z \in \Omega$. 

\end{proof}

\bibliographystyle{alpha}
\bibliography{complex_kob}

\end{document}